\documentclass[11pt]{article}

\usepackage{amsmath,amssymb,amsthm,color}
\usepackage[shortlabels]{enumitem}
\usepackage{graphicx}
\usepackage{refcount}
\usepackage{authblk}

\newcommand{\conv}{\mathrm{conv}}

\renewcommand{\Re}{\mathbb{R}}

\newcommand{\sB}{\mathfrak{B}}
\newcommand{\sC}{\mathfrak{C}}
\newcommand{\bbN}{\mathbb{N}}
\newcommand{\td}{\tilde{\delta}}
\renewcommand{\vert}{\mathrm{Vert}}

\newcommand{\sfrac}[2]{\mbox{$\frac{#1}{#2}$}}
\newcommand{\half}{\sfrac{1}{2}}

\newcommand{\mat}[1]{\left[ \begin{matrix} #1 \end{matrix} \right]}
\newcommand{\introcounter}[1]{\setcounterref{introthm}{#1}%
\addtocounter{introthm}{-1}}

\newtheorem{introthm}{Theorem}
\newtheorem{thm}{Theorem}
\newtheorem{prop}[thm]{Proposition}
\newtheorem{cor}[thm]{Corollary}

\newtheorem{lem}[thm]{Lemma}

\title{Diversities and the Generalized Circumradius}

\author[1]{David Bryant}
\author[2]{Katharina T. Huber}
\author[2]{Vincent Moulton}
\author[3]{Paul F.~Tupper} 
\affil[1]{\small Department of Mathematics and Statistics, University of Otago, Dunedin, New Zealand.  email: {\tt david.bryant@otago.ac.nz} phone: {\tt +64 34797889}. (Corresponding author)}
\affil[2]{\small School of Computing Sciences, University of East Anglia, NR4 7TJ, Norwich, United Kingdom}
\affil[3]{\small Department of Mathematics, Simon Fraser University, Burnaby, British Columbia V5A 1S6, Canada}

\date{\today}

\begin{document}
\maketitle


\begin{abstract}
    The {\em generalized circumradius} of a set of points $A \subseteq \Re^d$ with respect to a convex body $K$ equals the minimum value of $\lambda \geq 0$ such that a translate of $\lambda K$ contains $A$. Each choice of $K$ gives a different function on the set of bounded subsets of $\Re^d$; we characterize which functions can arise in this way. Our characterization draws on the theory of {\em diversities}, a recently introduced generalization of metrics from functions on pairs to functions on finite subsets. We additionally investigate functions which arise by restricting the generalized circumradius to a finite subset of $\Re^d$. We obtain elegant characterizations in the case that $K$ is a simplex or parallelotope.
\end{abstract}

The {\em circumradius} of a set of points in the plane is the radius of the smallest circle containing them. The concept is key to optimal containment and facility location problems, including a classic problem studied by Sylvester \cite{Jahn17,Sylvester57}, since the center of the smallest enclosing circle minimizes the maximum distance to any of the points.

The {\em generalized circumradius} replaces the plane with $\Re^d$ and the circle or ball with a general {\em convex body}, that is a compact, convex set with non-empty interior. For a convex body $K$ in $\Re^d$ and bounded $A \subseteq \Re^d$ we say that the generalized circumradius of $A$ with respect to $K$ is
\[R(A,K) = \inf\{\lambda \ge 0 : A \subseteq \lambda K+x \mbox{ for some $x \in \Re^d$}\}.\]
In other words, $R(A,K)$ equals the minimal amount that $K$ must be scaled so that a translate covers $A$ (see Fig.~\ref{MinkowskiExample}). The set $K$ is called the {\em kernel}. See \cite{BrandenbergKonig13,BrandenbergKonig15,Gonzalez-MerinoJahnEtal19,Jahn17b} for properties and inequalities related to the generalized circumradius and \cite{BrandenbergRoth09,Gritzmann93,GritzmannKlee92,GritzmannKlee94} for computational results.

\begin{figure}[htb]
\begin{center}
    \includegraphics[width=0.8\textwidth]{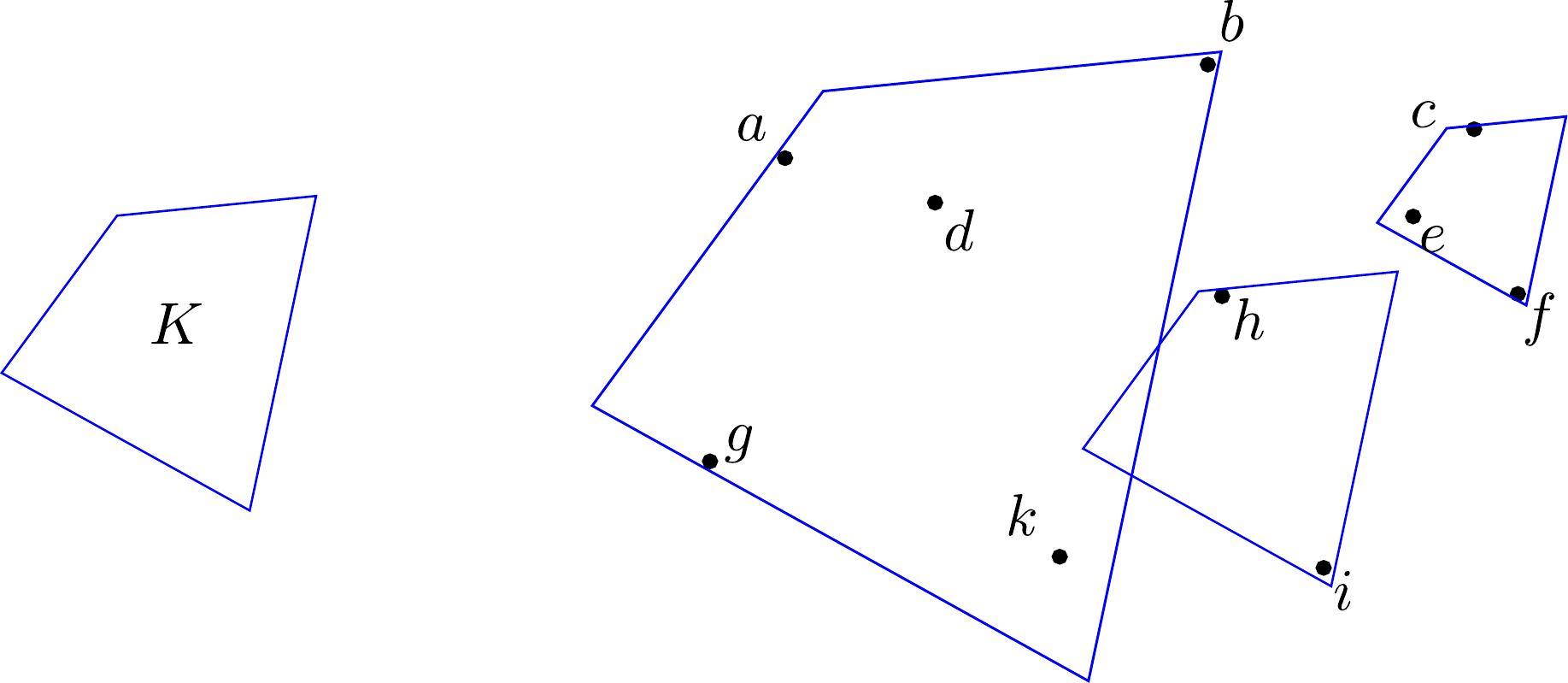}
\end{center}
\caption{ \label{MinkowskiExample} \small An example of the generalized circumradius. In this example $R(\{a,b,d,g,k\},K) = 2$ since $2K$ is the smallest scaled version of $K$ which can be translated to cover $\{a,b,d,g,k\}$. Similarly, $R(\{c,e,f\},K) = 0.6$, since $0.6K$ is the smallest scaled version of $K$ which can be translated to cover $\{c,e,f\}$, and $R(\{h,i\},K) = 1$.}
\end{figure}

Our motivation for studying the generalized circumradius comes from connections with {\em diversity theory}. A (mathematical) diversity is an extension of the idea of a metric space. Instead of assigning values just to pairs of objects, a diversity assigns values to all finite sets of objects. More formally, a {\em diversity} is a pair $(X,\delta)$ where $X$ is a set and $\delta$ a function from finite subsets of $X$ to $\Re_{\geq 0}$ such that, for $A,B,C$ finite, 
\begin{enumerate}
    \item[(D1)] $\delta(A)=0$ if and only if $|A| \leq 1$.
    \item[(D2)] If $B \neq \emptyset$ then $\delta(A \cup C) \leq \delta(A \cup B) + \delta(B \cup C)$.
\end{enumerate}
Diversities were introduced in \cite{BryantTupper12}. A consequence of (D1) and (D2) is that diversities are {\em monotonic}, that is, if $A \subseteq B$ then $\delta(A) \le \delta(B)$. Furthermore, $\delta$ restricted to pairs satisfies the definition of a metric; we call this the metric {\em induced} by $(X,\delta)$. We say that a diversity $(X,\delta)$ is {\em finite} if $X$ is finite. Note that on occasion we use the term `diversity' to refer to the function $\delta$ rather than the pair $(X,\delta)$. 

Many well-known functions on sets are diversities. Examples include
\begin{enumerate}
    \item The diameter of a set
    \item The length of a shortest Steiner tree connecting a set
    \item The mean width of a set
    \item The length of the shortest traveling salesman tour through a set
    \item The $L_1$ diversity (in $\Re^d$)
    \[\delta(A)= \delta_1(A) = \sum_{i=1}^d \max\{|a_i-a'_i|:a,a' \in A\}\]
    \item The circumradius of a set. 
\end{enumerate}
We show that the generalized circumradius is also a diversity:

\introcounter{thm:GCRdiversity}
\begin{introthm} Let $K$ be a convex body in $\Re^d$. If we define $\delta(A) = R(A,K)$ for all finite $A$ then $(\Re^d,\delta)$ is a diversity.
\end{introthm}

In reference to the concept of a {\em Minkowski norm} we say that a diversity $(\Re^d,\delta)$  is a {\em Minkowski diversity} if there is a convex body $K$ such that $\delta(A) = R(A,K)$ for all finite $A \subset \Re^d$.

Theorem~\ref{thm:GCRdiversity} connects the generalized circumradius to a growing and varied literature on diversity theory. The first diversity paper \cite{BryantTupper12} described diversity analogs to the metric tight span and metric hyperconvexity, leading to new results in analysis and fixed point theory \cite{KirkShahzad14,Piatek14,Poelstra13}. It was shown in \cite{BryantTupper14} that results of \cite{LinialLondonEtal95} and others on  embedding of finite metrics in $L_1$ can be extended to diversities, with potential algorithmic gains. There is a direct analog of the Urysohn metric space \cite{BryantNiesEtal17} for diversities and work on diversity theory within model theory \cite{BryantNiesEtal21,Hallback20}, lattice theory \cite{BryantFelipeEtal20,Toledo-Acosta16}, image analysis \cite{Dokania16}, machine learning \cite{KomodakisPawan-KumarEtal16} and phylogenetics \cite{BryantCioica-LichtEtal21,BryantTupper12,Steel14,Steel16}.

In this paper we are mainly concerned with characterizations and embeddings for Minkowski diversities---what characterizes these diversities  and which finite diversities can be embedded into a Minkowski diversity. Such embeddings (possibly  with  distortion) should in future provide valuable graphical representations of diversities in addition to algorithmic and computational tools. 

Regarding characterization we prove the following result.
A real-valued function $f$ on subsets of $\Re^d$ is said to be {\em sublinear} if 
\begin{align*}
    f(A+B) & \leq f(A) + f(B) & \mbox{for all $A,B$,}\\
    f(\lambda A) & = \lambda f(A) & \mbox{for all $A$ and $\lambda \geq 0$},
\end{align*}
where $A+B$ denotes the Minkowski sum.

\introcounter{GCRcharacterize}
\begin{introthm}
Let $(\Re^d,\delta)$ be a diversity. Then $\delta$ is a Minkowski diversity if and only if $\delta$ is sublinear and for all finite $A,B$ there are $a,b \in \Re^d$ such that 
\[ \delta( (A + a) \cup (B + b) ) \leq \max\{\delta(A),\delta(B)\}.\]
\end{introthm}

We also show that the last result extends beyond diversities to functions defined on bounded subsets of $\Re^d$.  
\introcounter{GCRcharacterize2}
\begin{introthm}
Let $f$ be a function on bounded subsets of $\Re^d$. Then there is a convex body $K$ such that $f(A) = R(A,K)$ for all bounded $A$ if and only if $f$ is sublinear, monotonic, and  $f$ restricted to finite subsets is a Minkowski diversity.
\end{introthm}

Having characterized which diversities on $\Re^d$ are Minkowski diversities, we turn to the  more difficult problem of characterizing which diversities can be isometrically embedded into Minkowski diversities. An {\em isometric embedding} of a diversity $(X,\delta_1)$ into a diversity $(\Re^d,\delta_2)$ is a map $\phi:X \rightarrow \Re^d$ such that $\delta_1(A) = \delta_2(\phi(A))$ for all finite $A \subseteq X$.  If there is an isometric embedding from a diversity $(X,\delta_1)$ into a Minkowski diversity $(\Re^d,\delta_K)$ for some $d$ and some convex body $K \subseteq \Re^d$ then we say that $(X,\delta_1)$ is  {\em Minkowski-embeddable}.

We provide a complete characterization of Minkowski-embeddability for diversities which are finite and {\em symmetric}, meaning that the diversity of a set is determined by a function of its cardinality.

\introcounter{GCRSymmetric}
\begin{introthm} 
Let $(X,\delta)$ be a finite symmetric diversity. Then $(X,\delta)$ is Minkowski-embeddable if and only if 
         \begin{equation}
              \frac{\delta(A \setminus \{a\})}{\delta(A)} \geq \frac{|A|-2}{|A|-1} \label{symmMink}
          \end{equation}
         for all $A \subseteq X$ with $|A| \geq 2$, $a \in A$.
\end{introthm}

A consequence of the theorem is that, even though any diversity on three elements is Minkowski-embeddable, there exist diversities on four elements which are {\em not}.

We then investigate which finite diversities $(X,\delta_1)$ can be embedded into the diversity $(\Re^d,\delta_2)$ when $\delta_2$ is a Minkowski diversity with kernel $K$ restricted to a particular class. A diversity $(X,\delta)$ is a {\em diameter diversity} if $\delta(A) = \max\{\delta(\{a,a'\}):a,a' \in A\}$ for all finite $A \subseteq X$, see \cite{BryantTupper12}. The following characterization for when $K$ is a cube (or non-degenerate transform of a cube) follows from an observation of \cite{BrandenbergKonig13}.

\introcounter{DCMDiameter}
\begin{introthm}
A finite diversity $(X,\delta)$ can be embedded in a Minkowski diversity with kernel equal to some parallelotope if and only if $(X,\delta)$ is a diameter diversity.
\end{introthm}

The case when $K$ is a simplex is more complex. We say that a finite diversity $(X,\delta)$ is of {\em negative type} if 
\[\sum_{A \neq \emptyset} \sum_{B \neq \emptyset} x_{A} x_{B} \delta(A \cup B)  \leq 0\]
for all zero-sum vectors $x$ indexed by non-empty subsets of $X$.  Diversities of negative type are analogous to metrics of negative type, and several of the properties of metrics of negative type extend to diversities of negative type, see  \cite{WuBryantEtal19}. For negative-type diversities we prove the following characterization.

\introcounter{NegativeType}
\begin{introthm}
A finite diversity $(X,\delta)$ can be embedded in a Minkowski diversity with kernel equal to some simplex if and only if $(X,\delta)$ has negative type.
\end{introthm}

Significantly, the set of diversities with negative type is quite large, with the same dimension as the set of diversities on $X$. The result shows that the class of Minkowski-embeddable diversities is even larger, potentially opening up possibilities for quite general theoretical and algorithmic results.

\section{The generalized circumradius}

In this section we collect together a number of fundamental results about the generalized circumradius. We begin with several  observations from \cite{Jahn17} (Proposition 3.2). Let $\conv(A)$ denote the convex hull of $A$.

\begin{prop} \label{jahn1} Let $A,A',B$ be bounded subsets of  $\Re^d$,  let $K,K'$ be convex bodies in $\Re^d$. Then
\begin{enumerate}
\item[(a)] If $A \subseteq A'$ and $K' \subseteq K$ then $R(A,K) \leq R(A',K')$.
\item[(b)] If $\conv(A) = \conv(A')$ then $R(A,K) = R(A',K)$.
\item[(c)] $R(A+B,K) \leq R(A,K) + R(B,K)$ with equality if  $B = \alpha K$ for some $\alpha \geq 0$.
\item[(d)] $R(A+x,K+y) = R(A,K)$  for all $x,y \in \Re^d$.
\item[(e)] $R(\alpha A, \beta K) = (\alpha/\beta) R(A,K)$ for $\alpha,\beta>0$. 
\end{enumerate}
\end{prop}

An indirect consequence of Helly's theorem (see e.g. \cite{DanzerGrunbaumEtal63}, Section 6.2) is that for bounded $A \subseteq \Re^d$ and a convex body $
K$ we can find a small subset $A' \subseteq A$ such that $|A'| \leq d+1$ and $R(A',K) = R(A,K)$. The following more general result forms one part of Theorem 1.2 in \cite{BrandenbergKonig13}.

\begin{prop} \label{coreset}
Suppose that $A \subset \Re^d$ is bounded and $K$ is a convex body. For all $\epsilon \geq 0$ there exists $A' \subseteq A$ such that $|A'| \leq \left \lceil \frac{d}{1+\epsilon} \right \rceil + 1$ and
\[R(A',K) \leq R(A,K) \leq (1+\epsilon) R(A',K).\]
\end{prop}

Note that for particular choices of $K$ there can be much smaller bounds on $|A'|$. For example, when $K$ is the Euclidean ball in $\Re^d$ we have for all bounded $A \subset \Re^d$ and $\epsilon>0$  that there is a subset $A' \subseteq A$ such that $R(A,K) \leq  (1+\epsilon) R(A',K)$ and 
\[|A'| \leq \left \lceil \frac{1}{2\epsilon+\epsilon^2} \right \rceil + 1.\]
This bound is independent of the dimension $d$.

We will make use of the following general property for Minkowski addition which is established during the proof of Theorem 4.1 in  \cite{BrandenbergKonig13}. Let $A \subseteq \Re^d$ be any  set with cardinality $k+1$, $k \ge 2$,  and zero sum.  Then 
\[A \subseteq \frac{k}{(k+1)(k-1)} \sum_{a \in A} \conv(A \setminus \{a\}).\]
Combining this observation with Proposition~\ref{jahn1} (d), (b) and (c) we  have
\begin{prop} \label{BrandenbergBound}
Suppose $A \subset \Re^d$, $|A|=k+1$ and $K$ is a convex body. Then 
\[R(A,K) \leq     \frac{k}{(k+1)(k-1)} \sum_{a \in A} R(A \setminus \{a\},K).\]
\end{prop}

For $C \subseteq \Re^p \times \Re^q$ we define the two projection operators
\begin{align*}
\pi_1(C) & = \{c_1 \in \Re^p : (c_1,c_2) \in C \mbox{ for some $c_2$}\} \\
\pi_2(C) & = \{c_2 \in \Re^q : (c_1,c_2) \in C \mbox{ for some $c_1$} \}.
\end{align*}
The following result will be useful for questions regarding embeddings.

\begin{prop} \label{times1}
Let $A$ be a bounded subset of $\Re^p \times \Re^q$. If $B$ and $C$ are convex bodies in $\Re^p$ and $\Re^q$ respectively, then 
\[ R(A,B \times C) = \max\{R(\pi_1(A),B), R(\pi_2(A),C)\}.\]
\end{prop}
\begin{proof}
Suppose $\lambda = R(A,B \times C)$ and that 
\[A  + (a_1,a_2) \subseteq \lambda (B \times C).\]
Applying $\pi_1$ and $\pi_2$ to both sides gives
\[\pi_1(A) + a_1 \subseteq \lambda B\]
and 
\[\pi_2(A) + a_2 \subseteq \lambda C.\]
Hence $ \max\{R(\pi_1(A),B), R(\pi_2(A),C)\} \leq R(A,B \times C)$. 

Conversely, if $\lambda =  \max\{R(\pi_1(A),B), R(\pi_2(A),C)\}$, then there are $a_1,a_2$ such that 
\begin{align*}
\pi_1(A) + a_1 & \subseteq \lambda B \\
\pi_2(A) + a_2 & \subseteq \lambda C.
\end{align*}
We then have
\begin{align*}
A + (a_1,a_2) & \subseteq \pi_1(A) \times \pi_2(A) + (a_1,a_2) \\
& = (\pi_1(A) + a_1) \times (\pi_2(A) + a_2) \\
& \subseteq \lambda (B \times C)
\end{align*}
so that $R(A,B \times C) \leq \max\{R(\pi_1(A),B), R(\pi_2(A),C)\}$.
\end{proof}

\section{Characterization of Minkowski diversities}

We begin this section by proving 
the first main result connecting the theory of generalized circumradii with  diversity theory.

\begin{thm} \label{thm:GCRdiversity} Let $K$ be a convex body in $\Re^d$. If $\delta(A) = R(A,K)$ for all finite $A$, then $(\Re^d,\delta)$ is a diversity.
\end{thm}

\begin{proof}
Clearly $\delta(A) \geq 0$ for all $A$ and $\delta(A) = R(A,K) = 0$ if and only if $|A| \leq 1$. Hence (D1) holds. By Proposition~\ref{jahn1} (a), $R(A,K)$ is monotonic in $A$. 

Suppose $A$ and $B$ are finite subsets of $\Re^d$ and $x \in A \cap B$. Then $(A-x) = (A-x) + 0 \subseteq (A-x)+(B-x)$ and $(B-x) = 0 + (B-x)  \subseteq (A-x)+(B-x)$.  Hence 
\[(A-x) \cup (B-x) \subseteq (A-x) + (B-x),\]
and so by Proposition~\ref{jahn1} (d) and (c),
\begin{align*}
\delta(A \cup B) &= \delta( (A \cup B) - x ) \\
&= \delta((A-x) \cup (B-x)) \\
& \leq \delta(A-x) + \delta(B-x) \\
&= \delta(A) + \delta(B).
\end{align*}
This fact, together with monotonicity, implies the diversity triangle inequality (D2).

\end{proof}

In the rest of this section we focus on
characterizing which diversities on $\Re^d$ can be obtained in this way (Theorem~\ref{GCRcharacterize}).

Let $\sB$ denote the Euclidean unit ball in $\Re^d$. The {\em Hausdorff distance} on (bounded) subsets of $\Re^d$ is given by
\[d_H(A,B) = \inf\{\lambda: A \subseteq B + \lambda \sB \mbox{ and } B \subseteq A + \lambda \sB \},\]
see \cite{Schneider14}. We note that $d_H$ becomes a metric when restricted to compact sets.

A diversity $(\Re^d,\delta)$ is {\em sublinear} if 
$\delta$ is a sublinear function, 
$\delta(A+B) \leq \delta(A) + \delta(B)$ and $\delta(\lambda A) = \lambda \delta(A)$ for all finite $A,B$ and non-negative $\lambda$. 
Note that finite sublinear diversities are closely related to {\em set-norms} \cite[Def. 2.1]{croitoru2010set}. The $L_1$ diversity $(\Re^d,\delta_1)$, as introduced earlier, is an example of a sublinear diversity. Here, 
 \[ \delta_{1}(A) = \sum_{i=1}^d \max\{|a_i - a'_i|:a,a' \in A\}, \]
 and $\delta_1$ is sublinear since,
 given $\lambda \geq 0$ and finite subsets $A,B \subseteq \Re^d$ we have
 \begin{align*}
     \delta_{1}(\lambda A) & =  \sum_{i=1}^d \max\{|\lambda a_i - \lambda a'_i|:a,a' \in A\}  \\
     & = \lambda \delta_1(A), \\ \intertext{ and }
     \delta_1(A + B) & = \sum_{i=1}^d \max\{|a_i+b_i- a'_i-b'_i|:a,a' \in A, b,b' \in B \} \\
     & \leq \sum_{i=1}^d \max\{|a_i- a'_i|:a,a' \in A\} + \max\{|b_i - b'_i|:b,b' \in B\} \\
     & = \delta_1(A) + \delta_1(B).
 \end{align*}

Sublinearity has several important consequences for diversities.  

\begin{prop} \label{subDprop}
Let $(\Re^d,\delta)$ be a sublinear diversity. Then
\begin{enumerate}
\item[(a)] $\delta$ is translation invariant: $\delta(A+x) = \delta(A)$ for all $x$.
\item[(b)] $\delta$ is determined by the convex hull: if $A,B$ finite in $\Re^d$ and $\conv(A) = \conv(B)$, then $\delta(A) = \delta(B)$.
\item[(c)] $\delta$ is  Lipschitz continuous with respect to the Hausdorff metric.
\end{enumerate}
\end{prop}
\begin{proof}
(a) By (D1) and sublinearity, we have 
\begin{align*}
\delta(A+x) &\leq \delta(A) + \delta(\{x\}) = \delta(A), \mbox{ and } \\ 
\delta(A)& \leq \delta(A+x) + \delta(\{-x\}) = \delta(A+x).
\end{align*}
(b) Let $B' \subseteq B$ be a maximal subset of $B$ such that $\delta(A \cup B') = \delta(A)$. Suppose that there is $b \in B \setminus B'$. As $b  \in \conv(B) = \conv(A)$ there are non-negative $\{\lambda_a\}_{a \in A}$ with unit sum such that 
\[b= \sum_a \lambda_a a \]
and hence
 \[ b \in \sum_a \lambda_a A .\]
 As 
 \[A \cup B' \subseteq \sum_a \lambda_a (A \cup B') \]
 we have
 \[\delta(A \cup B' \cup \{b\} ) \leq  \delta\Big( \sum_a \lambda_a (A \cup B') \Big) \leq \sum_a \lambda_a \delta(A \cup B') = \delta(A \cup B')= \delta(A).\]
 This contradicts the choice of $B'$. Hence $B' = B$ and $\delta(A \cup B) = \delta(A)$. Exchanging $A$ and $B$ in this argument gives $\delta(A \cup B) = \delta(B)$. Hence $\delta(A)=\delta(B)$.\\
(c) Let $V$ denote any finite set with convex hull $\conv(V)$ containing the unit ball $\sB$. One example is $V = 
\{-1,1\}^d$.  Let $\kappa = \delta(V)$  and suppose that $d_H(A,B) = \lambda$.
From the definition of $d_H$ we have $A \subseteq B +\lambda \sB$ and $B  \subseteq A + \lambda \sB$ so there are finite sets $A',B' \subseteq \sB$ such that
\[ A \subseteq B + \lambda A' \mbox{  and  }	B  \subseteq A + \lambda B'.\]
As $\sB \subseteq \conv(V)$, we have $\conv(A' \cup V) \!=\! \conv(B' \cup V) \!=\! \conv(V)$. By sublinearity, monotonicity of $\delta$ and part (b) we therefore have
\[ \delta(A)  \leq  \delta(B) + \lambda \delta(V) \mbox{ and } \delta(B)  \leq \delta(A) + \lambda \delta(V) \]
giving $|\delta(A) - \delta(B) | \leq \kappa d_H(A,B)$.
\end{proof}

Our next theorem gives a complete characterization of Minkowski diversities. A key idea in the proof of the theorem is that we can extend a sublinear diversity $(\Re^d,\delta)$ to a function on convex bodies in $\Re^d$. More specifically, given a sublinear diversity $\delta$, define the 
function $\td$ on the set  of convex bodies in $\Re^d$ by setting 
\[\td(P) = \delta(\vert(P))\]
for all polytopes $P$ with vertex set $\vert(P)$ and 
\[\td(K) = \ \lim_{n \rightarrow \infty} \td(P_n)\]
for any convex body $K$ and sequence $P_1,P_2,\ldots$ of polytopes converging under the Hausdorff metric to $K$.  

\begin{lem}\label{help}
Given a sublinear diversity $\delta$, the function $\td$ is well-defined and Lipschitz continuous.
\end{lem}
\begin{proof}
Since $\td$ is Lipschitz on the set of polytopes, it is uniformly continuous on the same set, and so can be uniquely extended to a continuous function on the closure of that set \cite[Prob. 13, Ch. 4]{rudin1976principles}, the convex bodies. An expression for $\td$ can then be obtain for convex bodies using the limits of sequences of polytopes as above, and this gives that the extension has the same Lipschitz constant.
\end{proof}

Using this observation, we now prove the main theorem for this section.

 \begin{thm}\label{GCRcharacterize}
Let $(\Re^d,\delta)$ be a diversity. Then $(\Re^d,\delta)$ is a Minkowski diversity if and only if $\delta$ is sublinear and for all finite $A,B$ there are $a,b \in \Re^d$ such that 
\begin{equation} \label{ABcup}
 \delta( (A + a) \cup (B + b) ) \leq \max\{\delta(A),\delta(B)\}.
 \end{equation}
\end{thm}
\begin{proof}
Suppose that $(\Re^d,\delta)$ is a Minkowski diversity, so that there is a convex body $K \subseteq \Re^d$ such that $\delta(A) = R(A,K)$ for all finite $A \subseteq \Re^d$. Sublinearity is given by Proposition~\ref{jahn1} parts (c) and (e). Given finite $A$ and $B$, there are $a,b \in \Re^d$ such that $A \subseteq \delta(A) K - a$ and $B \subseteq \delta(B)K - b$. Hence 
\[(A +a) \cup (B+b) \subseteq \max\{\delta(A),\delta(B)\} K\]
and  
\begin{align*}
    \delta((A +a) \cup (B+b)) & = R((A +a) \cup (B+b),K) \\
    & \leq R(\max\{\delta(A),\delta(B)\} K,K) \\
    & = \max\{\delta(A),\delta(B)\} R(K,K) \\
    & = \max\{\delta(A),\delta(B)\}.
\end{align*}

Now suppose that $(\Re^d,\delta)$ is sublinear and satisfies \eqref{ABcup} for all finite $A,B$.  
Let 
\[\rho = \max\left\{\frac{\|x\|}{\delta(\{0,x\})}: x \neq 0\right\} = \max\left\{\frac{1}{\delta(\{0,x\})}: \|x\|=1\right\}.\]
Then $0<\rho<\infty$ and $\|x-y\| \leq \rho \delta(\{x,y\})$  for all $x,y$.
(Here and below we use $\| \cdot \|$ to denote Euclidean norm).

We  show that for any convex bodies $L,K \subseteq \Re^d$ and the 
	function  $\td$ in Lemma~\ref{help} there is some $x \in \Re^d$ such that
	\begin{equation} \label{eqn:inequality_lemma}
 \td(\conv(K \cup (L+x))) \leq \max\{\td(K),\td(L)\}.
 \end{equation}
Suppose that $K_1,K_2,\ldots$ is a sequence of finite subsets of $K$ such that $\conv(K_n) \rightarrow K$ and $L_1,L_2,\ldots$ is a sequence of finite subsets of $L$ such that $\conv(L_n) \rightarrow L$.  Applying (\ref{ABcup}) and translation invariance of $\delta$ we have that for each $n \in \bbN$ there is an $x_n \in \Re^d$ such that 
\[\delta(K_n \cup (L_n + x_n)) \leq \max\{\delta(K_n),\delta(L_n)\}.\]
Hence, for all $n$,
\begin{equation} \label{eqn:Hausdorff_limit}
\td ( \conv(K_n \cup (L_n+x_n))) \leq \max\{ \td( \conv(K_n) ), \td( \conv(L_n))\}.
\end{equation}

We show that the sequence $x_n$ has a convergent subsequence. 
First note that since $K$ and $L$ are bounded, convergence of $\conv(K_n)$ and $\conv(L_n)$ to them in the Hausdorff metric implies that the union of all these sets is bounded, and hence  the sequence $\max\{\delta(K_n),\delta(L_n)\}$
is bounded. Choose $k_n \in K_n$ and $\ell_n \in L_n$ for each $n$, which are then bounded over all $n$.
 We then have 
 \[\|x_n - (k_n - \ell_n) \| = \|(x_n + \ell_n)-k_n \| \leq \rho  \delta(K_n \cup (L_n + x_n)) \leq \max\{\delta(K_n),\delta(L_n)\}.\]
 So the set $\{x_1,x_2,x_3,\ldots\}$ is bounded. Let $x_{i_1},x_{i_2},\ldots$ be a convergence subsequence and let $x \in \Re^d$ be its limit.

Since $\conv(L_n) \rightarrow L$ and $x_n \rightarrow x$, $\conv(L_n +x_n) \rightarrow L + x$. So, $\conv(K_n) \cup \conv(L_n + x_n)  \rightarrow K \cup (L +x)$
and 
\[
\conv(K_n \cup (L_n + x_n)) = \conv(\conv(K_n) \cup \conv(L_n + x_n)) \rightarrow \conv( K \cup (L+x)).
\]
Taking the limit as $n\rightarrow \infty$ of \eqref{eqn:Hausdorff_limit} and using the continuity of $\td$ gives 
\eqref{eqn:inequality_lemma}
which proves the claim.

Let $\sC$ denote the set of convex bodies 
\[\sC = \{A \subseteq \Re^d: \td(A) \leq 1\}.\]
The set $\sC$ is closed under the Hausdorff metric and both volume and $\td$ are continuous with respect to the Hausdorff metric \cite[Sec 1.8]{Schneider14}. 
It follows that there is some $K \in \sC$ such that the volume of $K$ is at least as large as any other element in $\sC$. The convex body $K$ is necessarily inclusion-maximum: if $K$ was a proper subset of some $K' \in \sC$ then the volume of $K'$ would be strictly greater than the volume of $K$. \\

  We claim that $R(A,K)=1$ for all finite $A$ such that $\delta(A) = 1$.\\

Let $A$ be finite with $\delta(A)=1$. By \eqref{eqn:inequality_lemma} there is  $x \in \Re^d$ such that \[\td(\conv(K \cup (\conv(A)+x))) \leq \max\{\td(K),\td(\conv(A))\} = 1.\]
We therefore have $\conv(K \cup (\conv(A)+x)) \in \sC$. As $K$ is set inclusion-maximum in $\sC$ and $K \subseteq \conv(K \cup (\conv(A)+x)) $ we have
\[K = \conv(K \cup (\conv(A)+x))\]
and so $A + x \subseteq K$ and $R(A,K) \leq 1$. On the other hand, if there is some  $\lambda\geq 0$ and $z \in \Re^d$ such that $A \subseteq \lambda K +z$ then 
\[\delta(A) = \td(\conv(A)) \leq \td(\lambda K + z) = \lambda \td(K) = \lambda,\]
showing that $R(A,K) \geq 1$. Hence $R(A,K)=1$, as claimed.

It follows that $\delta(A)=R(A,K)$ when $\delta(A)=1$.

More generally, suppose  $\delta(A)=d$. The case $d=0$ is straightforward, as then $|A|\leq 1$. If $d>0$ then, by sublinearity, $\delta(A)=d \delta(\frac{1}{d}A) = d R(\frac{1}{d} A,K) = R(A,K)$.
\end{proof}

We note that there are diversities on $\Re^d$ which are sublinear, but do not satisfy property \eqref{ABcup} in Theorem~\ref{GCRcharacterize} (and are hence not Minkowski diversities). For example, consider the $L_1$ diversity in the plane $(\Re^2,\delta_1)$. We saw above that $L_1$ diversities are sublinear but if 
\[A = \left\{ \mat{0\\0},\mat{1\\0} \right\} \mbox{ and } B = \left\{ \mat{0\\0},\mat{0\\1} \right\} \]
then for any $a,b \in \Re^2$ we have 
\[\delta_1((A+a) \cup (B+b) ) \geq 2 > 1 = \max\{\delta_1(A),\delta_1(B)\}.\]
~\\

\begin{figure}[tbp]
\begin{center}
    \includegraphics[width=0.7\textwidth]{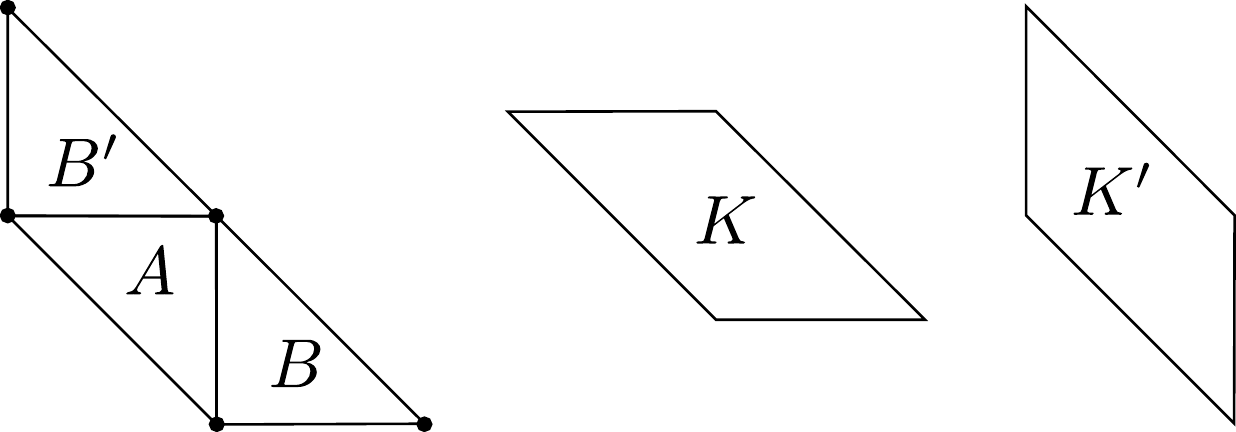}
\end{center}
\caption{\label{MinkowskiNonconvex} The sets $A,B,B'$ in the proof of Proposition~\ref{NonConvex} with (translated) kernels $K$ and $K'$. }
\end{figure}

The set of diversities on $\Re^d$, and indeed the the set of sublinear diversities on $\Re^d$, are both convex. Condition~\eqref{ABcup} in Theorem~\ref{GCRcharacterize} suggests that the set of Minkowski diversities on $\Re^d$ is not convex, as we now confirm.

\begin{prop} \label{NonConvex}
The set of Minkowski diversities on $\Re^d$, $d \geq 2$, is not convex.
\end{prop}
\begin{proof}
We first  establish this for $d=2$. Let 
\begin{align*}
A &= \left\{ \mat{1\\0},\mat{0\\1},\mat{1\\1} \right\}, \\
B &= \left\{ \mat{1\\0},\mat{2\\0},\mat{1\\1} \right\}, \\
B' & = \left\{ \mat{0\\1},\mat{1\\1},\mat{0\\2} \right\} .
\end{align*}
Let $K = \conv(A \cup B)$ and $K' = \conv(A \cup B')$. Let $(X,\delta)$ be the diversity on $\Re^2$ given by $\delta(Y) = \half R(Y,K) + \half R(Y,K')$. We will show that for all $a,b \in \Re^2$ 
\[\delta( (A + a) \cup (B + b) ) > \max\{\delta(A),\delta(B)\}, \]
from which it follows by Theorem~\ref{GCRcharacterize} that  $(X,\delta)$ is not a Minkowski diversity. 

First note that since $\delta$ is translation invariant, we can assume $a=0$. Now, note that
\[R(A,K) = R(A,K') = R(B,K) = R(B,K') = 1\]
so that 
\[\max\{\delta(A),\delta(B)\}  = 1.\]

 If $b=0$ then $R(A \cup (B+b),K) = 1$, otherwise $R(A \cup (B+b),K) > 1$.
 If $b = \mat{-1\\1}$ then $R(A \cup (B+b),K') = 1$, otherwise $R(A \cup (B+b),K') > 1$. 
 Hence we have 
\[	\delta( (A + a) \cup (B + b) )  = \half R(A \cup (B+b) ,K) + \half R(A \cup (B+b),K')  > 1\]
even though $\max\{\delta(A),\delta(B)\} = 1$.

For $d>2$, in the above argument we replace $K$ and $K'$ with their product with the unit hypercube in $\Re^{d-2}$, and append $d-2$ zeros to the elements in $A$ and $B$ and to $x$.
\end{proof}

\section{Characterizing the general circumradius}

In this section, we characterize functions $f$ for which there is a convex body $K$ such that $f(A) = R(A,K)$ for all bounded $A$, noting that in the previous section we only considered finite $A$.  The main idea behind the proof is to show that a sublinear, monotonic function is Hausdorff continuous, after which the result follows almost immediately from Theorem~\ref{GCRcharacterize}.

\begin{lem}\label{helper}
Let $f$ be a sublinear, monotonic function on bounded subsets of $\Re^d$. Then $f$ is Hausdorff continuous. 
\end{lem}
\begin{proof}
Let $\sB$ denote the Euclidean unit ball in $\Re^d$, let $b = f(\sB)$ and let $\epsilon >0$. Then 
for any bounded $A,B$ such that $d_H(A,B) \leq \frac{\epsilon}{b}$ we have
\[f(A) \leq f\left(B + \frac{\epsilon}{b} \sB\right) \leq f(B) + \epsilon \]
and
\[f(B) \leq f\left(A + \frac{\epsilon}{b} \sB\right) \leq f(A) + \epsilon.\]
\end{proof}

\begin{thm} \label{GCRcharacterize2}
Let $f$ be a function on bounded subsets of $\Re^d$. Then there is a convex body $K$ such that $f(A) = R(A,K)$ for all bounded $A$ if and only if $f$ is sublinear, monotonic, and  $f$ restricted to finite subsets is a Minkowski diversity.
\end{thm}
\begin{proof}
Necessity follows from the arguments used for Theorem~\ref{GCRcharacterize}.

Suppose that $f$ is sublinear and monotonic, and $f$ restricted to finite subsets is a Minkowski diversity. By  Lemma~\ref{helper},  $f$ is Hausdorff continuous. Let $A$ be a bounded subset of $\Re^d$. Then $f(\overline{A}) = f(A)$, where $\overline{A}$ denotes the topological closure of $A$. 

By Theorem~\ref{GCRcharacterize} there is a convex body $K$ such that $f(B) = R(B,K)$ for any finite $B$. Given any natural number $n \ge 1$ there is a finite cover of $\overline{A}$ by balls of radius $\frac{1}{n}$. Let $A_n$ denote the set of centers of those balls, so that $A_n \rightarrow \overline{A}$. We then have 
\[f(A) = f(\overline{A}) = \lim_{n \rightarrow \infty} f(A_n) = \lim_{n \rightarrow \infty} R(A_n,K) = R(A,K).\]
\end{proof}

\section{Embedding finite diversities} \label{Section:FiniteEmbed}

In this section we consider an (isometric) embedding problem: when 
can a given finite diversity $(X,\delta)$ be embedded in a Minkowski diversity? 
There is a long history in mathematics regarding embedding of finite {\em metrics} into standard spaces. Perhaps best known is the characterization due to Cayley and Menger of when a finite metric can be embedded into Euclidean space \cite{Blumenthal70,DezaLaurent97}. The theory of metric embeddings forms the basis of many methods for multi-dimensional scaling, an approximate low-dimensional embedding designed specifically for data reduction and representation. Approximate embeddings have proven exceptionally useful for algorithm design and approximations (e.g.~\cite{LinialLondonEtal95}), work that has a direct analog in the mathematics of diversities \cite{BryantTupper14}.

To discuss embeddings, it is convenient to consider a slight generalization of diversities. A {\em semimetric} is a bivariate, symmetric map $d$ on $X$ that vanishes on the diagonal and satisfies the triangle inequality, but where we allow $d(x,y) = 0$ even when $x \neq y$, $x,y \in X$ (so, in particular a metric is a semimetric). Similarly, a pair $(X,\delta)$ is a {\em semidiversity} if it satisfies (D2) and the following slightly weaker version of (D1)
\begin{enumerate}
	\item[(D1')] $\delta(A)=0$ if  $|A| \leq 1$.
\end{enumerate}

\noindent We say that a (semi)diversity $(X,\delta)$ is {\em Minkowski-embeddable} if for some $d$ there is a map $\phi:X \rightarrow \Re^d$  and a convex body $K$ in $\Re^d$ such that 
\[ \delta(A) = R(\phi(A),K)\]
for all finite $A \subseteq X$.

For the rest of this section we shall focus on the embedding problem
for symmetric diversities, where a (semi)diversity $(X,\delta)$ 
is {\em symmetric} if $\delta(A) = \delta(B)$ whenever $|A| = |B|$, $A,B \subseteq X$, that is, the value of $\delta$ on a set depends only upon the cardinality of the set, see \cite{BryantTupper16}. We shall characterize when a finite symmetric diversity is Minkowski-embeddable. As a corollary we also show that not every diversity is  Minkowski-embeddable

We start with some utility results on embeddings. For convenience, for the rest of this section we shall assume that $X$ is a finite set. Note that if $(X,\delta_1)$ and $(X,\delta_2)$ are two semidiversities then $(X,\delta_1 \vee \delta_2)$ denotes the semidiversity with 
\[(\delta_1 \vee \delta_2)(A) = \max\{\delta_1(A),\delta_2(A)\}\]
for all $A \subseteq X$. To see that this is a semidiversity, note that for all $A,B,C$ with $B \neq \emptyset$ we have without loss of generality, 
\[(\delta_1 \vee \delta_2)(A \cup C) = \delta_1(A \cup C)\]
so that
\[(\delta_1 \vee \delta_2)(A \cup C)  \leq \delta_1(A \cup B) + \delta_1(B \cup C) \leq (\delta_1 \vee \delta_2)(A \cup B)+(\delta_1 \vee \delta_2)(B \cup C).\]

\begin{prop}~\label{prop:utility}
\begin{enumerate}
\item[(a)] Let $(X,\delta_1)$ and $(X,\delta_2)$ be Minkowski-embeddable semidiversities, and $\lambda > 0$. Then both $(X,\lambda \delta_1)$ and $(X,\delta_1 \vee \delta_2)$ are Minkowski-embeddable.
\item[(b)] Suppose that $K$ is a convex body in $\Re^d$ and $\phi:\Re^d \rightarrow \Re^d$ is a non-degenerate affine map. Then for all $A$ we have $R(\phi(A),\phi(K)) = R(A,K)$. Hence if there is an isometric embedding from $(X,\delta)$ into $(\Re^d,K)$ then there is also an isometric embedding from $(X,\delta)$ into $(\Re^d,\phi(K))$. 
\item[(c)] If $(X,\delta)$ is Minkowski-embeddable and $|A| = k+1$, $k \ge 2$, then
\begin{equation}
\delta(A) \leq \frac{k}{(k+1)(k-1)} \sum_{a \in A} \delta(A \setminus \{a\}).
\end{equation}
\end{enumerate}
\end{prop}
\begin{proof}
\noindent (a) 
There are maps $\phi_1:X \rightarrow \Re^p$ and $\phi_2:X \rightarrow \Re^q$ and convex bodies $K_1 \subset \Re^p$ and $K_2 \subset \Re^q$ such that $\delta_1(A) = R(\phi_1(A),K_1)$ and $\delta_2(A) = R(\phi_2(A),K_2)$ for all $A \subseteq X$. 


Then for all $A \subseteq X$ we have 
\[\lambda \delta_1(A) = R(\phi_1(A),\lambda^{-1} K_1)\]
and by Proposition~\ref{times1} 
\[(\delta_1 \vee \delta_2)(A) = R(\phi_1(A) \times \phi_2(A), K_1 \times K_2).\]
\noindent (b) If there are $\lambda \geq 0$ and $x$ such that $A + x\subseteq \lambda K$ then $\phi(A) + \phi(x) \subseteq \lambda \phi(K)$ so $R(\phi(A),\phi(K)) \leq R(A,K)$. Applying the inverse map gives equality.\\
\noindent (c) There is a map $\phi:X \rightarrow \Re^d$ and a convex body $K$ such that $R(\phi(A),K) = \delta(A)$. Applying Proposition~\ref{BrandenbergBound} gives the result.
\end{proof}

We now consider Minkowski-embeddability for a few key examples of symmetric diversities.

\begin{prop} \label{prop:special}
\begin{enumerate}
\item[(a)] The diversity $(X,\delta)$ with $\delta(A) = 1$ for all $A \subseteq X$ with $|A|>1$ is Minkowski-embeddable.
\item[(b)]
The diversity $(X,\delta)$ with $\delta(A) = |A|-1$ for all non-empty $A \subseteq X$ is Minkowski-embeddable.
\item[(c)] Any diversity $(X,\delta)$ with $X =\{a,b,c\}$ and $\delta(\{a,b\}) = \delta(\{a,c\}) = \delta(\{b,c\}) = 1$  is Minkowski-embeddable. 
\end{enumerate}
\end{prop}
\begin{proof}
 (a)  Let $n = |X| - 1$. Let $K$ be the simplex with vertex set $V$ given by the standard basis vectors in $\Re^n$ together with $0$. Then for non-singleton subset $V' \subseteq V$ we have $R(V',K) = 1$. Hence any bijection from $X$ to $V$ gives an isometric embedding.\\
(b) Let $K$ be the same simplex as in (a) and now let $V$ be the vertex set of $-K$. The proof of Theorem 4.1 in \cite{BrandenbergKonig13} then gives $R(V',K) = |V'|-1$ for all non-empty subsets of $V' \subseteq V$. The result follows. \\
(c) Let $x = \delta(\{a,b,c\})$. As $(X,\delta)$ is a diversity, $1 \leq x \leq 2$. Let $(X,\delta_1)$ be the diversity on $X$ with $\delta_1(A) = 1$ for all non-singleton $A \subseteq X$ and let $(X,\delta_2)$ be the diversity on $X$ with $\delta_2(A) = |A|-1$ for all non-empty $A \subseteq X$. Then $(X,\delta_1)$ and $(X,\delta_2)$ are Minkowski-embeddable, and by Proposition~\ref{prop:utility} (a) so is 
\[(X,\delta) = (X, \delta_1 \vee (\sfrac{x}{2} \delta_2)).\]
\end{proof}

We now give an exact characterization for when a finite symmetric diversity is Minkowski-embeddable.

     \begin{thm} \label{GCRSymmetric}
         Let $(X,\delta)$ be a finite symmetric diversity. Then $(X,\delta)$ is Minkowski-embeddable if and only if 
         \begin{equation}
              \frac{\delta(A \setminus \{a\})}{\delta(A)} \geq \frac{|A|-2}{|A|-1} \label{symmMink}
          \end{equation}
         for all $A \subseteq X$ with $|A| \geq 2$, $a \in A$.
     \end{thm}
       \begin{proof}
      Suppose that $(X,\delta)$ is Minkowski-embeddable,  that $A \subseteq X$, $|A| \geq 2$ and $a \in A$.
      
      If $|A| = 2$ then \eqref{symmMink} holds trivially. If $|A| > 2$ then by Proposition~\ref{prop:utility} (c) 
      \begin{align*}
          \delta(A) &\leq \frac{|A| - 1} {|A| (|A|-2)} \sum_{x \in A} \delta(A \setminus \{x\}) \\
          &= \frac{|A|-1}{|A|(|A|-2) } |A| \delta(A \setminus \{a\}),
      \end{align*}
      where the last line follows since $\delta$ is symmetric. Hence \eqref{symmMink} holds.
      
      Conversely, suppose that \eqref{symmMink} holds for all $A \subseteq X$ such that $|A| \geq 2$ and $a \in A$. As $(X,\delta)$ is symmetric, $\delta$ is monotonic and there is an increasing function $f:\mathbb{Z} \rightarrow \Re_{\ge 0}$ such that $\delta(A) = f(|A| - 1)$ for all non-empty $A \subseteq X$. Note that for each $k =2, \ldots, |X|$, choosing an $A \subseteq X$ with $|A|=k$ and substituting $A$ into  \eqref{symmMink}  gives
      \begin{equation}
          \frac{f(k-2)}{f(k-1)} \geq \frac{k-2}{k-1}, \label{MinkowskiBound}
      \end{equation}
      a fact that we shall use later on in the proof.

      Let $X = \{x_1,x_2,\ldots,x_n\}$ and for each $m \leq n$ define $X_m = \{x_1,x_2,\ldots,x_m\}$.  
      We will use induction on $m$ to show that for each $m \leq n$ the symmetric diversity $(X,\delta)$ restricted to $X_m$ is Minkowski-embeddable. This is clearly the case when $m \leq 2$.

      Suppose that $2<m \leq n$ and that $(X,\delta)$ restricted to $X_{m-1}$ is Minkowski-embeddable.
      Then there is a map $\phi:X_{m-1} \rightarrow \Re^d$ for some $d$ and a convex body $K \subseteq \Re^d$ such that 
      \[\delta_K(\phi(A)) = \delta(A)\]
      for all $A \subseteq X_{m-1}$.
      We now 
      define a collection $\delta^{(i)}$ of (semi)diversities on $X_m$, $1\le i \le m$. 
      
      First, 
      for each $i = 1,2,\ldots,m-1$  define the map $\phi^{(i)}:X_m \rightarrow \Re^d$ by $\phi^{(i)}(x) = \phi(x)$ if $x \in X_{m-1}$ and $\phi^{(i)}(x_m) = \phi(x_i)$. We then define 
      $(X_m,\delta^{(i)})$ by 
      \[\delta^{(i)}(A) = \delta_K(\phi^{(i)}(A)).\]
      Since $(\Re^d,\delta_K)$ is a diversity, 
      $(X_m,\delta^{(i)})$ satisfies (D1') and (D2), and so 
      $(X_m,\delta^{(i)})$ is a Minkowski-embeddable semidiversity. Moreover, the definitions of $\phi$ and $\phi^{(i)}$ give  
      \[ \delta^{(i)}(A) = \begin{cases} f(|A|-2) & \mbox{ if $x_i,x_m \in A$ } \\ f(|A|-1) & \mbox{ otherwise.} \end{cases}\]
      
      Second, let $(X_m,\delta^{(m)})$ be the diversity defined by setting
      \[\delta^{(m)}(A) = \frac{|A|-1}{m-1} f(m-1)\]
      for all non-empty $A \subseteq X_m$. This is  Minkowski-embeddable by Proposition~\ref{prop:special} (b)  and  Proposition~\ref{prop:utility} (a). 
      
      We now claim that
      \begin{equation} \delta(A) = \max\{\delta^{(i)}(A):i=1,\ldots,m\} \label{deltaIsMax} \end{equation}
      holds for all $A \subseteq X_m$. Proving this claim will complete
      the proof of the theorem by induction since each (semi)diversity $\delta^{(i)}$ is Minkowski-embeddable, and hence
      by Proposition~\ref{prop:utility} (a) $(X_m,\delta)$ is Minkowski-embeddable. 
      
      When $|A| \leq 1$, \eqref{deltaIsMax} holds trivially, as all relevant quantities are zero. Suppose $|A| \geq 2$. Three cases may hold: 
      \begin{itemize}
          \item $x_i \not \in A$ for some $i=1,\ldots,m-1$; then $\delta^{(i)}(A) = f(|A|-1)$.
          \item $x_m \not \in A$; then $\delta^{(i)}(A) = f(|A|-1)$ for all $i=1,\ldots,m-1$.
          \item $A = X_m$; then $\delta^{(i)}(A) = f(|A|-2)$ for all $i=1,\ldots,m-1$.
      \end{itemize}
      Hence
      \[\max\{\delta^{(i)}(A):i=1,2,\ldots,m-1\} = \begin{cases} f(|A|-1) & \mbox{ if $A \neq X_m$;} \\ f(|A|-2) & \mbox{ if $A=X_m$.} \end{cases}\]
      
      We now consider $\delta^{(m)}$ and $A \subseteq X_m$. If $A = X_m$ then 
      \[\delta^{(m)}(A)  = \frac{(|A|-1)}{(m-1)} f(m-1)  = f(m-1).\]
       Otherwise suppose $1<|A| < m$ and so
      \begin{align*}
          \delta^{(m)}(A) & = \frac{(|A|-1)}{(m-1)} f(m-1) ,\\ 
          & =  \frac{(|A|-1)}{(m-1)}  \left( \prod_{k=|A|+1}^m \frac{f(k-1)}{f(k-2)} \right) f(|A|-1), \\
          & \leq  \frac{(|A|-1)}{(m-1)}  \left( \prod_{k=|A|+1}^m \frac{k-1}{k-2} \right) f(|A|-1) & \mbox{ by \eqref{MinkowskiBound}, }\\
          & =  f(|A|-1).
      \end{align*}
      Hence if $A \neq X_m$, 
      \[\max\{\delta^{(i)}(A):i=1,2,\ldots,m\}  = \max\{f(|A|-1),\delta^{(m)}(A)\} = f(|A|-1),\]
      while if $A= X_m$ then since $f$ is increasing
      \[\max\{\delta^{(i)}(A):i=1,2,\ldots,m\}  = \max\{f(|A|-2),\delta^{(m)}(A)\} = f(|A|-1).\]      
We conclude that $\delta(A) = \max\{\delta^{(i)}(A):i=1,\ldots,m\}$
      for all $A \subseteq X_m$ which completes the proof of \eqref{deltaIsMax} and also the theorem.
  \end{proof}

By considering the diversity on $X = \{a,b,c,d\}$ with 
\[\delta(A) = \begin{cases} 2, & |A| = 4, \\ 1, & 2 \leq |A| \leq 3, \\ 0, & \mbox{otherwise,} \end{cases} \]
and taking $A=\{b,c,d\}$ in Theorem~\ref{GCRSymmetric} we immediately obtain

\begin{cor}\label{4pt}
	There exists a diversity a set of four elements that is not Minkowski-embeddable.
\end{cor}

We show later (Corollary~\ref{NegConsequences}) that {\em every} diversity on three elements is Minkowski-embeddable. 
 
\section{Parallelotopes and simplices}

We have shown that not every diversity is Minkowski-embeddable, and so 
the question now becomes one of characterizing which diversities are. 
In this section we characterize when we can embed
the diameter and negative-type diversities defined in the introduction
in terms of Minkowski diversities having kernels equal to parallelotopes and simplices, respectively.

We first consider diameter diversities.

\begin{thm} \label{DCMDiameter}
A finite diversity $(X,\delta)$ can be embedded in a Minkowski diversity with kernel equal to some parallelotope if and only if $(X,\delta)$ is a diameter diversity.
\end{thm}
\begin{proof}
First note that if there is some such embedding then $(X,\delta)$ is a diameter diversity by Proposition 3.4 of \cite{BrandenbergKonig13}.

Conversely, suppose that  $X = \{x_1,x_2,\ldots,x_n\}$ and 
 $(X,\delta)$ is a diameter diversity. Let $\phi:X \rightarrow \Re^n$ be the standard Fr\'echet embedding 
\[\phi:X \rightarrow \Re^n: y \mapsto (d(x_1,y), d(x_2,y), \ldots,d(x_n,y))\]
of the metric $d$ induced by $\delta$.
Then $d(x,y) = \| \phi(x) - \phi(y) \|_\infty$ for all $x,y \in X$. 

Let $K$ be the unit cube in $\Re^n$. For all $A \subseteq X$ we have 
\begin{align*}
\delta(A) & = \max_{a,b \in A} d(a,b) \\
& = \max_{a,b \in A} \| \phi(a) - \phi(b) \|_\infty \\
& = \max_i \max_{a,b \in A} | \phi(a)_i - \phi(b)_i | \\
& = R(\phi(A),K).
\end{align*}
\end{proof}

We now consider finite diversities of negative type,
the diversity analog of metrics of negative type \cite{DezaLaurent97}. 
Note that the cone of all diversities on a set $X$ of cardinality $n$ has dimension $2^{n} - (n+1)$, the number of subsets $A \subset X$ with $|A| \geq 2$, see \cite{WuBryantEtal19}. The set of diversities of negative type forms a cone of the same dimension, indicating that an appropriately chosen `random' diversity could have non-negative type with non-zero probability.


\begin{thm} \label{NegativeType}
A finite diversity $(X,\delta)$ can be embedded in a Minkowski diversity with kernel equal to some simplex if and only if $(X,\delta)$ has negative type.
\end{thm}

\begin{proof}
Let $n=|X|$. From Theorem 7 of \cite{WuBryantEtal19}, $(X,\delta)$ is negative-type if and only if it can be embedded in $(\Re^d,\delta_{neg})$ for some $d$ and 
\[\delta_{neg}(A) = \sum_{i=1}^d \max\{a_i: a \in A\} - \min\left\{\sum_{i=1}^d a_i:a\in A\right\}.\]

Define the polytope 
\begin{align*}
K &= -\mathrm{conv}(0,e_1,e_2,\ldots,e_d)\\
& = \{ x \leq  0 : \sum_{i=1}^d x_i \geq -1\}.
\end{align*}
Suppose $\lambda = R(A,K)$, so there is $z$ such that $A  \subseteq \lambda K + z$. Then for all $a \in A$ and all $i$ we have $a_i - z_i \leq 0$. Hence 
\[z_i \geq \max\{a_i:a \in A\}.\]
Let $a^* \in A$ minimize $\sum_{i=1}^d a^*_i$. Then $a^* \in \lambda K + z$ implies $\sum_{i=1}^d (a^*_i - z_i) \geq -\lambda$ and  so
\begin{align*}
\lambda & \geq \sum_{i=1}^d z_i  - \sum_{i=1}^d a^*_i \\ 
&\geq \sum_{i=1}^d \max\{a_i:a \in A\} - \min_{a \in A} \sum_{i=1}^d a_i \\
& = \delta_{neg}(A). 
\end{align*}
Now suppose $\lambda = \delta_{neg}(A)$. Let $z_i = \max\{a_i:a \in A\}$ so that $a_i-z_i \leq 0$ for all $a \in A$ and all $i=1,\ldots,d$. 
Furthermore
\begin{align*}
-\lambda & = -\sum_{i=1}^d z_i + \min\left\{\sum_{i=1}^d a_i:a\in A\right\} \\
& \leq -\sum_{i=1}^d z_i + \sum_{i=1}^d a_i  & \mbox{ for all $a \in A$}
\end{align*}
so that $\sum_{i=1}^d (a_i - z_i) \geq - \lambda$ and $a - z  \in \lambda K$ for all $a \in A$.  Hence $R(A,K) \leq \delta_{neg}(A)$.

We have shown that $(X,\delta)$ is of negative type if and only if it can be embedded into a Minkowski diversity with kernel equal to the particular simplex $K$. The theorem now follows from Proposition~\ref{prop:utility}(b) and the fact that every simplex in $\Re^d$ can be transformed into another by a non-degenerate affine map.

\end{proof}

The last theorem immediately implies that two further classes are Minkowski-embeddable.

\begin{cor} \label{NegConsequences}
If $\delta$ is diversity on three elements or a finite diversity that can be embedded in $L_1$, then $\delta$ is Minkowski-embeddable.
\end{cor}
\begin{proof}
All three-element diversities and $L_1$-embeddable diversities have negative type \cite{WuBryantEtal19}.
\end{proof}

\section{Open problems}

We have characterized diversities and functions defined by the generalized circumradius $R(A,K)$, and established preliminary results on embedding finite diversities into these spaces. Our results suggest several avenues for further investigation.

First, is there a complete characterization of when a finite diversity is Minkowski-embeddable? Indeed it is not even obvious which finite diversities can be embedded into {\em sublinear} diversities in $\Re^d$. 

A second related question is algorithmic in nature: Are there efficient algorithms 
	for determining whether or not a finite diversity can be embedded in some dimension?
	Interestingly, we note that for the classical case of a circumradius, even though 
	we do not know a characterization for Minkowski-embeddability,
	we are able to give an efficient algorithm for deciding embeddability (for bounded dimension):

\begin{prop}\label{prop:polynomial}
	Let $(X,\delta)$ be a finite diversity such that $\delta(A) = \max\{\delta(A'): A' \subseteq A, |A'| \leq d+1\}$ for all $A \subseteq X$. For fixed $d$, there is an algorithm which runs in polynomial time in $n=|X|$ to determine if $(X,\delta)$  is Minkowski-embeddable in $\Re^d$ with kernel  equal to the unit ball $\sB$.
\end{prop}
\begin{proof}
We begin with a useful observation. Suppose that there is an (unknown) embedding $\phi:X \rightarrow \Re^d$ such that $\delta(A) = \delta_{\sB}(\phi(A))$ for all $A \subseteq X$. Note that since the metric induced by $\delta_{\sB}$ is Euclidean so is the one induced by $\delta$. Let $\psi:X \rightarrow \Re^d$ be any map which preserves the metrics induced by $\delta$ and $\delta_{\sB}$. In addition, let
$f$ be the (unknown) isometry from $\psi(X)$ to $\phi(X)$ given by $f(\psi(x)) = \phi(x)$ for all $x \in X$.  As $\Re^d$ is a finite dimensional Hilbert space, $f$ can be extended to an isometry on the whole space $\Re^d$ (see, e.g., \cite{WellsWilliams12}, Theorem 11.4). Moreover, for any $A \subseteq X$ we have
\begin{align*}
    \delta_{\sB}(\psi(A))  & = 	 \inf\Big\{ \sup\{\|\psi(a)-x\|_2:a \in A\}: x \in \Re^d \Big\} \\
    & = \inf\Big\{ \sup \{\|f(\psi(a))-f(x)\|_2:a \in A\}: x \in \Re^d \Big\} \\
    & = \inf\Big\{\sup \{\|\phi(a)-y\|_2:a \in A\}: y \in \Re^d \Big\} \\
    & = \delta_{\sB}(\phi(A)) \\
    & = \delta(A).
\end{align*}
Hence, if  $(X,\delta)$ is Minkowski embeddable 
into $(\Re^d,\delta_{\sB})$, then the map $\psi$ gives one embedding.

We now present an algorithm for deciding whether or not $(X,\delta)$ is embeddable in $(\Re^d,\delta_{\sB})$:
\begin{enumerate}
    \item Decide whether or not the metric induced by $\delta$ on $X$ is Euclidean. If not, then $(X,\delta)$ cannot be embedded in $(\Re^d,\delta_{\sB})$. Else, 
    compute a (metric) embedding $\psi$ of $X$ in $\Re^d$ which preserves the induced metrics. 
    \item If $\delta_{\sB}(\psi(A)) = \delta(A)$ for all $A$ with $|A| \leq d+1$ then $(X,\delta)$ can be embedded in $(\Re^d,\delta_{\sB})$, otherwise $(X,\delta)$ cannot be embedded in $(\Re^d,\delta_{\sB})$.
\end{enumerate}

The correctness of this algorithm follows by the observation above. To see that it also runs in polynomial time in $n$ (for fixed $d$), note that
Step~1 can be computed in polynomial time in $n$ by the results in e.g. \cite[Section 6.2]{DezaLaurent97} or \cite[Theorem 2.1]{WellsWilliams12}, and that for Step~2 the definition of $\delta$ and Proposition~\ref{coreset} imply that to determine  whether $\delta_{\sB}(\psi(A)) = \delta(A)$ for all $A \subseteq X$ we need only check subsets $A$ with $|A|\leq d+1$.
\end{proof}

Another question is how to extend the embedding results to include distortion. Let $(X_1,\delta_1)$ and $(X_2,\delta_2)$ be two diversities. We say that a map $\phi:X_1 \rightarrow X_2$ has {\em distortion} $c$ if there are $c_1,c_2>0$ such that $c=c_1c_2$ and
\[\frac{1}{c_1} \delta_1(A) \leq \delta_2(\phi(A)) \leq c_2 \delta_1(A)\]
for all finite $A \subseteq X$. Continuing the program of \cite{LinialLondonEtal95}, it is shown in \cite{BryantTupper14} that bounds on the distortion of embeddings from diversities into $L_1$-diversities provide approximation algorithms for hypergraph generalizations of sparsest cut. It was shown in \cite{WuBryantEtal19} that there are finite diversities of metric type which cannot be embedded into $L_1$ without at least $\Omega(\sqrt{\log |X|})$ distortion. This bound therefore holds for Minkowski-embeddable diversities. In general, questions concerning distortion seem intricately connected with core sets of the generalized circumradius  \cite{BrandenbergKonig13}. 
 
Apart from potential algorithmic gains it would be good to explore embeddings with distortion for diversities into Minkowski diversities with low dimension, simply for their use in visualization and modelling of diversity type data. 

\subsection*{Acknowledgments}

We thank Pei Wu and Malcolm Jones for discussions and ideas which helped initiate this paper.  PT is supported by an NSERC (Canada) Discovery Grant. DB, KTH and VM thank the Royal Society for its support. We also thank the reviewers for their helpful comments.

\bibliographystyle{acm}
\bibliography{MinkowskiArXiv}

\end{document}